\documentclass[11pt]{article}

\usepackage[T1]{fontenc}
\usepackage[utf8]{inputenc}
\usepackage{lmodern}
\usepackage{amsmath,amssymb,amsthm,mathtools}
\usepackage{enumitem}
\usepackage{geometry}
\usepackage{hyperref}
\usepackage{xcolor}
\usepackage{booktabs}
\usepackage{microtype}

\hypersetup{
	colorlinks=true,
	linkcolor=blue,
	citecolor=blue,
	urlcolor=blue
}

\newtheorem{theorem}{Theorem}[section]
\newtheorem{proposition}[theorem]{Proposition}
\newtheorem{lemma}[theorem]{Lemma}
\newtheorem{corollary}[theorem]{Corollary}
\newtheorem{question}[theorem]{Question}
\theoremstyle{definition}
\newtheorem{definition}[theorem]{Definition}
\newtheorem{example}[theorem]{Example}
\theoremstyle{remark}
\newtheorem{remark}[theorem]{Remark}

\newtheorem*{maintheorem}{\textbf{Theorem}}

\newcommand{\Fp}{\mathbb F_p}
\newcommand{\Zp}{\mathbb Z_p}

\newcommand{\PD}{\mathrm{PD}}
\newcommand{\HH}{\mathrm H}
\newcommand{\cdp}{\operatorname{cd}}
\newcommand{\rk}{\operatorname{rk}}
\newcommand{\im}{\operatorname{im}}
\newcommand{\cl}{\operatorname{cl}}

\title{Canonical Kummerian Kernels in 1-Smooth\\Pro-\texorpdfstring{$p$}{p} Poincar\'e Duality Groups of Dimension Three}
\author{Marina Palaisti}
\date{}

\begin{document}
	\maketitle
	
\begin{abstract}
Let $(G,\theta)$ be a finitely generated 1-smooth oriented pro-$p$ group with torsion-free orientation, and suppose that $G$ is a Poincar\'e duality pro-$p$ group of dimension three. The orientation determines a canonical closed normal subgroup $K_\theta(G)$ contained in the Frattini subgroup, and for a Kummerian pair the quotient $G/K_\theta(G)$ is the maximal $\theta$-abelian quotient.
		
We prove that
\[K_\theta(G)\text{ finitely generated} \quad\Longrightarrow\quad K_\theta(G)=1.\]
Thus every non-$\theta$-abelian 1-smooth $\PD^3$ pro-$p$ group has an infinitely generated canonical kernel, where ``infinitely generated'' means not topologically finitely generated. The proof combines the normal-subgroup theorem for pro-$p$ $\PD^3$ groups with the classification of $p$-adic analytic 1-smooth pro-$p$ groups. We also obtain a cohomological description of the canonical extension from the Hochschild--Serre five-term sequence and Poincar\'e duality. Writing $K=K_\theta(G)$, $Q=G/K$, and $V=\HH^1(G,\Fp)$, we identify $\HH^1(K,\Fp)^Q$ with the kernel of the map $\bigwedge^2V\to V^*$ induced by the triple cup-product form of $G$. Consequently, whenever $K\neq1$, its first cohomology is infinite-dimensional, whereas its $Q$-invariant subspace is nonzero and finite-dimensional. We obtain exact dimension restrictions in small generator ranks and show that a non-$\theta$-abelian group with $2\leq d(G)\leq4$ is not Bloch--Kato. In the finitely generated-kernel case, $G$ is a three-dimensional uniform $\theta$-abelian group with exterior mod-$p$ cohomology.
	\end{abstract}
	
	
	~
	~
	
	
\section{Introduction}
	
The smoothness framework of De Clercq and Florence was introduced to isolate formal lifting properties modeled on Hilbert 90 and Galois cohomology \cite{DeClercqFlorence}. In the pro-$p$ setting, an orientation
\[ \theta:G\longrightarrow \Zp^\times\]
allows one to formulate 1-smoothness, also called 1-cyclotomicity. The Smoothness Conjecture predicts that 1-smoothness imposes strong quadraticity properties on mod-$p$ cohomology; in particular, a 1-smooth pro-$p$ group is expected to be weakly Bloch--Kato.
	
Several substantial classes are understood. Quadrelli proved the conjectural implication for finitely generated $p$-adic analytic pro-$p$ groups \cite{QuadrelliHHA} and established related results for solvable groups \cite{QuadrelliCMB}. Snopce and Zalesskii treated right-angled Artin pro-$p$ groups \cite{SnopceZalesskii}. Poincar\'e duality groups remain a particularly natural testing ground: the cohomology is rigid, while the subgroup structure becomes increasingly subtle with the dimension. Quadrelli explicitly asked whether 1-smooth Poincar\'e duality pro-$p$ groups are weakly Bloch--Kato \cite[Remark~5.10]{QuadrelliHHA}.
	
The dimension-two case is classical. Pro-$p$ $\PD^2$ groups are Demu\v{s}kin groups, and their canonical cyclotomic behavior is well understood. Dimension three is the first case in which one can combine the canonical subgroup attached to an orientation with a strong normal-subgroup theorem for the ambient Poincar\'e duality group.
	
For an oriented pro-$p$ group $(G,\theta)$, let $K_\theta(G)$ denote the closed normal subgroup generated by the orientation relations 
\[h^{-\theta(g)}ghg^{-1}, \qquad
g\in G,\ h\in\ker(\theta).\]
This subgroup satisfies \[K_\theta(G)\subseteq\ker(\theta)\cap\Phi(G).\]
When $(G,\theta)$ is Kummerian, the quotient
\[ G/K_\theta(G)\] is the maximal $\theta$-abelian quotient. Thus 1-smoothness does not merely provide some analytic quotient: it provides a canonical one, whose kernel is contained in the Frattini subgroup and whose quotient preserves the entire degree-one mod-$p$ cohomology of $G$.
	
The main result of the paper is the following finite-generation obstruction.
	
\begin{maintheorem} \label{thm:main-intro}
\textit{Let $(G,\theta)$ be a finitely generated 1-smooth oriented pro-$p$ group with torsion-free orientation. Suppose that $G$ is a pro-$p$ $\PD^3$ group. Then}
\[K_\theta(G)=1 \qquad\textit{or}\qquad K_\theta(G)\textit{ is infinitely generated}.\]
\textit{Equivalently, if $K_\theta(G)$ is finitely generated, then} $K_\theta(G)=1$.
\end{maintheorem}
	
Here and throughout, ``infinitely generated'' means not topologically finitely generated; equivalently, for a pro-$p$ group $K$ it means $d(K)=\dim_{\Fp}\HH^1(K,\Fp)=\infty$. The point is not that finite generation of the kernel is known a priori for a large class of nonanalytic examples. Rather, the theorem shows that finite generation cannot occur as an intermediate form of complexity. The canonical obstruction is either absent or genuinely infinite-rank. 
This gives a structural restriction on all non-$\theta$-abelian 1-smooth pro-$p$ $\PD^3$ groups. Both alternatives are realized: the $\theta$-abelian case consists of three-dimensional uniform groups, while cyclotomic extensions of Demu\v{s}kin groups provide non-$\theta$-abelian examples with free canonical kernels of infinite rank.

Two existing results are decisive. The first is the normal-subgroup theorem of Grunewald, Jaikin-Zapirain, Pinto and Zalesskii: a nontrivial finitely generated normal subgroup of infinite index in a pro-$p$ $\PD^3$ group is either procyclic with virtually Demu\v{s}kin quotient, or Demu\v{s}kin with virtually procyclic quotient \cite[Corollary~1.5]{GrunewaldEtAl}. The second is Quadrelli's theorem that a finitely generated $p$-adic analytic 1-smooth pro-$p$ group with torsion-free orientation is $\theta$-abelian \cite[Theorem~5.7]{QuadrelliHHA}. The new interaction is that the canonical Kummerian kernel is an intrinsic normal subgroup to which the $\PD^3$ alternative applies, while its quotient is already $\theta$-abelian and hence analytic. 

Although these ingredients are individually established, they arise in rather different parts of the literature and do not by themselves single out the canonical kernel considered here. The observation used in this paper is that, for a 1-smooth pair, $K_\theta(G)$ simultaneously lies in the Frattini subgroup, has a canonical $\theta$-abelian quotient, and preserves degree-one mod-$p$ cohomology. These three properties make the $\PD^3$ normal-subgroup dichotomy directly applicable to an intrinsic subgroup rather than to a chosen analytic quotient. To the best of our knowledge, the resulting finite-generation obstruction for $K_\theta(G)$ has not previously been recorded.
	
The finite-generation obstruction has further consequences. Since $K_\theta(G)$ has infinite index whenever it is nontrivial, a non-$\theta$-abelian example necessarily satisfies
\[d\bigl(K_\theta(G)\bigr)=\infty,	\qquad \cdp\bigl(K_\theta(G)\bigr)\leq2.\]
Moreover, the canonical Frattini extension
\[ 1 \longrightarrow K_\theta(G) \longrightarrow G \longrightarrow G/K_\theta(G) \longrightarrow1 \]
has an injective transgression
\[ \HH^1(K_\theta(G),\Fp)^{G/K_\theta(G)} \hookrightarrow \HH^2(G/K_\theta(G),\Fp).
\]
Since the quotient is uniform, the target is finite-dimensional and can be computed explicitly from degree one. Thus any nontrivial canonical kernel has infinite-dimensional first cohomology but only a finite-dimensional invariant part. In particular, the quotient cannot act trivially on the Frattini quotient of the kernel. 

When $K_\theta(G)$ is finitely generated, the conclusion is explicit: $G$ is a three-dimensional uniform $\theta$-abelian group. Hence 
\[ G\simeq\Zp^3 \] for the trivial orientation, while for a nontrivial orientation \[	G\simeq\Zp^2\rtimes_\theta\Zp\] with scalar action. Lazard's theorem then gives 
\[ \HH^\bullet(G,\Fp) \simeq \bigwedge\nolimits^\bullet\HH^1(G,\Fp) \]
as graded $\Fp$-algebras.
	
The paper is organized as follows. Section~\ref{sec:smooth} fixes the orientation conventions and develops the canonical Kummerian quotient used throughout the paper. Section~\ref{sec:pd3} recalls the relevant structure of pro-$p$ $\PD^3$ groups and establishes a finite-rank lemma for procyclic extensions of $p$-adic analytic groups. Section~\ref{sec:main} proves the finite-generation obstruction for $K_\theta(G)$. Section~\ref{sec:cohomological} derives further consequences from the Hochschild--Serre five-term sequence, including constraints on the invariant subspace of $\HH^1(K_\theta(G),\Fp)$. Section~\ref{sec:classification} describes explicitly the finitely generated-kernel case and classifies the resulting $\theta$-abelian $\PD^3$ groups. Finally, Section~\ref{sec:examples} discusses examples, the scope of the main result, and the structure that any remaining infinite-kernel case must satisfy.

\section{1-smooth orientations and the canonical quotient} \label{sec:smooth}
	
\subsection{Orientation conventions}
	
An \textit{oriented pro-$p$ group} is a pair $(G,\theta)$ consisting of a pro-$p$ group $G$ and a continuous homomorphism
\[ \theta:G\longrightarrow\Zp^\times. \]
Throughout the paper we work with torsion-free orientations in the standard pro-$p$ sense. Thus for odd $p$ we assume
\[ \im(\theta)\subseteq1+p\Zp,\]
and for $p=2$ we impose the standing convention
\[\im(\theta)\subseteq1+4\mathbb Z_2.\]
In particular, $\im(\theta)$ is torsion-free. This is the convention under which the Kummerian and analytic structure theorems used below are formulated.
	
Let $\Zp(\theta)$ denote the rank-one $\Zp$-module on which $G$ acts through $\theta$. Reduction modulo $p^n$ gives the discrete module $\Zp(\theta)/p^n$.
	
\begin{definition}
 An oriented pro-$p$ group $(G,\theta)$ is \textit{Kummerian} if, for every $n\geq1$, the natural reduction map
 \[\HH^1\!\left(G,\Zp(\theta)/p^n\right)\longrightarrow \HH^1(G,\Fp)\]
is surjective.
\end{definition}
\begin{definition}
An oriented pro-$p$ group $(G,\theta)$ is \textit{1-smooth} if every closed subgroup $H\leq G$, equipped with the restricted orientation $\theta|_H$, is Kummerian. \end{definition}
	
This is the formulation used in \cite{QuadrelliHHA}. In particular, 1-smoothness immediately implies that $(G,\theta)$ itself is Kummerian.

	\subsection{The canonical Kummerian kernel}
	
For an oriented pro-$p$ group $(G,\theta)$ define
\begin{equation}
K_\theta(G)=\cl\left\langle h^{-\theta(g)}ghg^{-1} 	\ \middle|\ g\in G,\ h\in\ker(\theta) \right\rangle. \label{eq:Kdef}\end{equation}
The subgroup $K_\theta(G)$ is closed and normal in $G$, and one has
\begin{equation}
[\ker(\theta),\ker(\theta)] \subseteq K_\theta(G)\subseteq \ker(\theta)\cap\Phi(G). \label{eq:Kbasic} \end{equation}
See \cite{EfratQuadrelli,QuadrelliWeigelBP}.
	
The quotient by $K_\theta(G)$ forces the orientation relations
\[ghg^{-1}=h^{\theta(g)} \qquad (g\in G,\ h\in\ker(\theta)).\]
The remaining issue is whether $\ker(\theta)/K_\theta(G)$ contains $p$-power torsion. The Kummerian condition is precisely what removes that obstruction in the finitely generated torsion-free setting. 

\begin{definition}
A torsion-free oriented pro-$p$ group $(G,\theta)$ is \textit{$\theta$-abelian} if
\[K_\theta(G)=1\]
and $\ker(\theta)$ is a free abelian pro-$p$ group.
\end{definition}
	
Let $I_\theta(G)$ be the isolator of $K_\theta(G)$ in $\ker(\theta)$. We record the precise consequence of Kummerianity that will be used throughout the paper. 

\begin{proposition} \label{prop:canonical}
Let $(G,\theta)$ be a finitely generated Kummerian oriented pro-$p$ group with torsion-free orientation. Then: 
\begin{enumerate}[label=\textup{(\roman*)}]
\item $\ker(\theta)/K_\theta(G)$ is a free abelian pro-$p$ group;
\item \[ I_\theta(G)=K_\theta(G);\]
\item with $\bar G=G/K_\theta(G)$ and the induced orientation $\bar\theta$, the pair $(\bar G,\bar\theta)$ is $\bar\theta$-abelian and is the maximal $\theta$-abelian quotient of $(G,\theta)$; 
\item $K_\theta(G)\subseteq\Phi(G)$.
	\end{enumerate}\end{proposition}
	
\begin{proof}
The equivalence between Kummerianity, freeness of $\ker(\theta)/K_\theta(G)$, and $\bar\theta$-abelianity of the quotient is \cite[Theorem~2.2]{QuadrelliHHA}. In the torsion-free Kummerian setting, the equality $I_\theta(G)=K_\theta(G)$ and the maximality of the resulting $\theta$-abelian quotient are given by \cite[Proposition~2.6 and Corollary~2.9]{QuadrelliWeigelBP}. The final inclusion is part of \eqref{eq:Kbasic}. \end{proof}
	
\begin{remark}
The distinction between $K_\theta(G)$ and its isolator is important in general. Proposition~\ref{prop:canonical} is the reason that no such distinction appears in the remainder of this paper: 1-smoothness implies Kummerianity, so the canonical maximal $\theta$-abelian quotient is genuinely $G/K_\theta(G)$. 
\end{remark}

\subsection{Degree-one cohomology and analyticity}
	
The Frattini inclusion has two immediate consequences that are central to the argument.
	
\begin{lemma} \label{lem:H1}
Let $(G,\theta)$ be as in Proposition~\ref{prop:canonical}, and put
\[K=K_\theta(G), \qquad\bar G=G/K.  \]
Then the quotient map induces an isomorphism
\[ G/\Phi(G) \simeq \bar G/ \Phi(\bar G). \]
Equivalently, inflation is an isomorphism
\[\inf: \HH^1(\bar G,\Fp) \overset{\sim}{\longrightarrow} \HH^1(G,\Fp).\]
In particular,
\[d(\bar G)=d(G).\]
\end{lemma}
	
\begin{proof}
Since $K\subseteq\Phi(G)$, one has \[ \Phi(\bar G)=\Phi(G)/K. \]Hence
\[\bar G/\Phi(\bar G) \simeq G /\Phi(G). \]
Taking continuous homomorphisms to $\Fp$ gives the degree-one inflation isomorphism. 
\end{proof}

\begin{proposition}
\label{prop:quotient-uniform}
Let $(G,\theta)$ be finitely generated and Kummerian with torsion-free orientation. Then
\[ \bar G=G/K_\theta(G)\]
is a finitely generated uniform pro-$p$ group. In particular, it is $p$-adic analytic and 
\[\dim(\bar G)=d(\bar G)=d(G).\]
Moreover, \begin{equation} \HH^\bullet(\bar G,\Fp)   \simeq \bigwedge\nolimits^\bullet\HH^1(\bar G,\Fp)
\label{eq:barGcohomology}\end{equation} as graded $\Fp$-algebras.
\end{proposition}
	
\begin{proof}
By Proposition~\ref{prop:canonical}, $(\bar G,\bar\theta)$ is $\bar\theta$-abelian. Proposition~4.2 of \cite{QuadrelliHHA} identifies torsion-free $\theta$-abelian pro-$p$ groups with locally uniform pro-$p$ groups. Since $\bar G$ is finitely generated, the subgroup $\bar G\leq\bar G$ is itself finitely generated; applying the definition of local uniformity to this subgroup shows explicitly that $\bar G$ is uniform. A uniform pro-$p$ group is $p$-adic analytic, and its analytic dimension equals its minimal number of generators. Lemma~\ref{lem:H1} gives $d(\bar G)=d(G)$. Finally, \eqref{eq:barGcohomology} is Lazard's cohomology theorem for uniform pro-$p$ groups \cite{Lazard}; see also \cite[(4.1)]{QuadrelliHHA}. \end{proof}

\section{Poincar\'e duality in dimension three} \label{sec:pd3}
	
We recall only those features of pro-$p$ Poincar\'e duality groups that enter the proof.
	
\begin{definition}
A pro-$p$ group $G$ is a Poincar\'e duality pro-$p$ group of dimension $n$, abbreviated $\PD^n$, if it has finite cohomological dimension $n$, satisfies the standard finiteness condition, and its continuous cohomology satisfies Poincar\'e duality in dimension $n$. 
\end{definition}

A pro-$p$ group of finite cohomological dimension has no nontrivial $p$-torsion. Subgroups of infinite index in a pro-$p$ $\PD^n$ group have cohomological dimension at most $n-1$. In low dimensions, pro-$p$ $\PD^1$ groups are isomorphic to $\Zp$, while pro-$p$ $\PD^2$ groups are Demu\v{s}kin groups. We shall also use the following immediate numerical consequence of Poincar\'e duality.
	
\begin{lemma}\label{lem:PDdimensions}
If $G$ is a pro-$p$ $\PD^3$ group, then
\[\dim_{\Fp}\HH^2(G,\Fp)=\dim_{\Fp}\HH^1(G,\Fp)=d(G).\]
\end{lemma}
	
\begin{proof}
Poincar\'e duality gives a perfect pairing between degree one and degree two mod-$p$ cohomology. The mod-$p$ dualizing character is trivial for a pro-$p$ group, since its image in $\Fp^\times$ has both $p$-power order and order dividing $p-1$ (and $\mathbb F_2^\times$ is trivial). Hence
\[\dim_{\Fp}\HH^2(G,\Fp)= \dim_{\Fp}\HH^1(G,\Fp).\]
Finally, $\dim_{\Fp}\HH^1(G,\Fp)=d(G)$ for every finitely generated pro-$p$ group.
\end{proof}
	
The decisive dimension-three input is the following theorem.
	
\begin{theorem}[Grunewald--Jaikin-Zapirain--Pinto--Zalesskii] \label{thm:GJPPZ}
Let $G$ be a finitely generated pro-$p$ $\PD^3$ group and let $1\neq N\trianglelefteq G$ be a finitely generated normal subgroup of infinite index. Then one of the following alternatives holds: \begin{enumerate}[label=\textup{(\roman*)}]
\item $N\simeq\Zp$ and $G/N$ is virtually Demu\v{s}kin;
\item $N$ is Demu\v{s}kin and $G/N$ is virtually $\Zp$.
\end{enumerate} \end{theorem}
	
\begin{proof} This is exactly \cite[Corollary~1.5]{GrunewaldEtAl}. The same pro-$p$ alternative is recalled in the later subgroup analysis of Castellano and Zalesskii \cite{CastellanoZalesskii}. \end{proof}
	
We shall also use the following elementary analytic-extension lemma. It is useful to isolate it because it is precisely what rules out the procyclic-kernel alternative in the main theorem.
	
\begin{lemma}
\label{lem:cyclic-extension}
Let
\[ 1\longrightarrow C \longrightarrow E \longrightarrow Q \longrightarrow 1 \]
be an exact sequence of finitely generated pro-$p$ groups with $C\simeq\Zp$ and $Q$ $p$-adic analytic. Then $E$ has finite rank and is therefore $p$-adic analytic.
\end{lemma}
	
\begin{proof}
Let $H\leq E$ be any closed subgroup. There is an exact sequence 
\[ 1\longrightarrow H\cap C \longrightarrow H \longrightarrow HC/C \longrightarrow  1.\]
The subgroup $HC/C$ is closed in $Q$: it is the continuous image of the compact group $H$ in the Hausdorff group $Q$. Since $C\simeq\Zp$, we have
\[ d(H\cap C)\leq1. \]
Since $Q$ is $p$-adic analytic, it has finite rank. If $r=\rk(Q)$, then
\[d(HC/C)\leq r.\]Lifting generators of $HC/C$ to $H$ and adjoining at most one generator of $H\cap C$ yields
\[ d(H)\leq r+1.\]
This bound holds for every closed subgroup $H\leq E$, so $E$ has finite rank. By the finite-rank characterization of finitely generated $p$-adic analytic pro-$p$ groups \cite[Theorem~5.1]{QuadrelliHHA}, $E$ is $p$-adic analytic.
\end{proof}

\section{The finite-generation obstruction} \label{sec:main}
	
We now prove the main result in the form used later.
	
\begin{theorem}\label{thm:main}
Let $(G,\theta)$ be a finitely generated 1-smooth oriented pro-$p$ group with torsion-free orientation. Suppose that $G$ is a pro-$p$ $\PD^3$ group. Then 
\[K_\theta(G)=1 \qquad\text{or}\qquad K_\theta(G)\text{ is infinitely generated}. \] \end{theorem}
	
	\begin{proof}
Set
\[K=K_\theta(G), \qquad \bar G=G/K. \]
Since $(G,\theta)$ is 1-smooth, it is Kummerian. Therefore Propositions~\ref{prop:canonical} and~\ref{prop:quotient-uniform} apply: $\bar G$ is a uniform, hence torsion-free $p$-adic analytic, pro-$p$ group and 
\begin{equation}
d(\bar G)=d(G).\label{eq:dmain}
\end{equation}
		
Assume for contradiction that $K$ is nontrivial and finitely generated. We first note that $K$ has infinite index in $G$. Indeed, $K\subseteq\Phi(G)$, so \eqref{eq:dmain} implies that $\bar G$ is nontrivial. Since $\bar G$ is uniform, it is torsion-free; hence a nontrivial $\bar G$ is infinite. Thus
\[[G:K]=\infty.\]
Theorem~\ref{thm:GJPPZ} now leaves two possibilities.
		
\smallskip

\noindent\textit{Case 1: $K$ is Demu\v{s}kin and $\bar G$ is virtually $\Zp$.} Because $\bar G$ is $p$-adic analytic and has an open subgroup isomorphic to $\Zp$, its analytic dimension is one. Since $\bar G$ is uniform, 
\[d(\bar G)=\dim(\bar G)=1.\]
Equation \eqref{eq:dmain} gives $d(G)=1$. Hence $G$ is procyclic. The group $G$ has finite cohomological dimension, so it has no nontrivial $p$-torsion; therefore \[G\simeq\Zp. \]This contradicts $\cdp(G)=3$.
		
\smallskip
\noindent\textit{Case 2: $K\simeq\Zp$ and $\bar G$ is virtually Demu\v{s}kin.}
The quotient $\bar G$ is already $p$-adic analytic by Proposition~\ref{prop:quotient-uniform}. Lemma~\ref{lem:cyclic-extension} applied to 
\[ 1\longrightarrow K\longrightarrow G\longrightarrow\bar G \longrightarrow 1\]
shows that $G$ itself is $p$-adic analytic.
		
Now apply Quadrelli's analytic 1-smoothness theorem \cite[Theorem~5.7]{QuadrelliHHA}: a finitely generated $p$-adic analytic 1-smooth pro-$p$ group with torsion-free orientation is $\theta$-abelian. Therefore \[K_\theta(G)=1,\]contradicting the assumption that $K$ is nontrivial.
		
Both alternatives lead to contradictions. Hence a finitely generated $K_\theta(G)$ must be trivial.
	\end{proof}
	
\begin{remark}
The proof uses no classification of arbitrary 1-smooth $\PD^3$ groups. Its force comes from the fact that the canonical Kummerian quotient simultaneously has three properties: its kernel is normal and lies in $\Phi(G)$, its quotient is analytic, and degree-one cohomology is preserved. These are exactly the features that make the normal-subgroup theorem effective. 
\end{remark}
	
\begin{corollary} \label{cor:nonanalytic}
Under the hypotheses of Theorem~\ref{thm:main}, if $G$ is not $\theta$-abelian, then
\[d\bigl(K_\theta(G)\bigr)=\infty.\]
If, in addition, $K_\theta(G)\neq1$, then
\[\cdp\bigl(K_\theta(G)\bigr)\leq2.\]
\end{corollary}
	
\begin{proof}
For a finitely generated Kummerian pair, Theorem~2.2 of \cite{QuadrelliHHA} shows that $K_\theta(G)=1$ is equivalent to $\theta$-abelianity. Thus a non-$\theta$-abelian group has $K_\theta(G)\neq1$, and Theorem~\ref{thm:main} implies that the kernel is not finitely generated. Equivalently, $d(K_\theta(G))=\infty$.

As in the proof of Theorem~\ref{thm:main}, the quotient $G/K_\theta(G)$ is a nontrivial torsion-free pro-$p$ group, so the kernel has infinite index. An infinite-index subgroup of a pro-$p$ $\PD^3$ group has cohomological dimension at most two; see, for example, \cite{CastellanoZalesskii}.
\end{proof}

\section{Cohomological consequences of the canonical extension} \label{sec:cohomological}
	
The main theorem identifies the size of a possible nontrivial canonical kernel. The Hochschild--Serre sequence gives an additional constraint on how such a kernel can sit inside $G$. Let 
\begin{equation}
1\longrightarrow K \longrightarrow G \overset{\pi}{\longrightarrow}Q \longrightarrow 1 \label{eq:canonical-extension}
\end{equation}
be the canonical extension, where
\[K=K_\theta(G), \qquad Q=G/K_\theta(G). \]
By Lemma~\ref{lem:H1}, inflation
\[ \HH^1(Q,\Fp)\longrightarrow\HH^1(G,\Fp)\]
is an isomorphism. The five-term exact sequence therefore begins as
\[	0\longrightarrow \HH^1(Q,\Fp) \overset{\inf}{\longrightarrow}\HH^1(G,\Fp) \longrightarrow\HH^1(K,\Fp)^Q \overset{\operatorname{trg}}{\longrightarrow} \HH^2(Q,\Fp) \overset{\inf}{\longrightarrow} \HH^2(G,\Fp).\]
The first map is already an isomorphism, so exactness immediately gives the following.

\begin{proposition}\label{prop:canonical-transgression}
Let $(G,\theta)$ satisfy the hypotheses of Theorem~4.1, and put 
\[K=K_\theta(G),\qquad Q=G/K,\;\; V=H^1(G,\mathbb F_p),\;\; d=\dim_{\mathbb F_p}V.\]
Choose an identification
\[\operatorname{tr}:H^3(G,\mathbb F_p)\xrightarrow{\sim}\mathbb F_p\]
and define the alternating trilinear form
\[T_G:V\times V\times V\longrightarrow\mathbb F_p,\;\; T_G(x,y,z)=\operatorname{tr}(x\smile y\smile z).\]
Let
\[\mu_G:\bigwedge\nolimits^2 V\longrightarrow V^*\]
be the linear map determined by
\[\mu_G(x\wedge y)(z)=T_G(x,y,z).\] Then there is an isomorphism of $\mathbb F_p$-vector spaces 
\[H^1(K,\mathbb F_p)^Q\simeq\ker(\mu_G).\]
Consequently,
\[\dim_{\mathbb F_p}H^1(K,\mathbb F_p)^Q= \binom d2-\operatorname{rank}(\mu_G), \]
and in particular
\[\max\left\{0,\binom d2-d\right\} \leq \dim_{\mathbb F_p}H^1(K,\mathbb F_p)^Q\leq \binom d2. \]
If $K\ne1$, then
\[H^1(K,\mathbb F_p)^Q\ne0.\]
Moreover, in this case the following sharper conclusions hold:
\begin{enumerate}
\item If $d=2$, then
\[\dim_{\mathbb F_p}H^1(K,\mathbb F_p)^Q=1.\]
\item If $d=3$, then $T_G=0$, and
\[\dim_{\mathbb F_p}H^1(K,\mathbb F_p)^Q=3.\]
\item If $d=4$, then
\[\dim_{\mathbb F_p}H^1(K,\mathbb F_p)^Q \in\{3,6\}.\]
\end{enumerate}
\end{proposition}

\begin{proof}
By Lemma~2.6, inflation induces an isomorphism
\[H^1(Q,\mathbb F_p)\xrightarrow{\sim} H^1(G,\mathbb F_p).\]
The Hochschild--Serre five-term exact sequence for
\[ 1\longrightarrow K\longrightarrow G \longrightarrow Q\longrightarrow 1 \]
therefore identifies the transgression with an injection
\[H^1(K,\mathbb F_p)^Q \hookrightarrow H^2(Q,\mathbb F_p)\]
whose image is precisely the kernel of degree-two inflation. Hence
\[H^1(K,\mathbb F_p)^Q \simeq \ker\left( H^2(Q,\mathbb F_p)\xrightarrow{\inf}H^2(G,\mathbb F_p) \right).\]
	
By Proposition~2.7, $Q$ is uniform and $d(Q)=d$. Lazard's theorem gives
\[H^\bullet(Q,\mathbb F_p) \simeq \bigwedge\nolimits^\bullet H^1(Q,\mathbb F_p). \]Using the degree-one inflation isomorphism, we obtain an identification 
\[ H^2(Q,\mathbb F_p) \simeq\bigwedge\nolimits^2 V.\]
Under this identification, degree-two inflation is the cup-product map
\[ c_G:\bigwedge\nolimits^2V \longrightarrow H^2(G,\mathbb F_p),\;\;x\wedge y\longmapsto x\smile y. \]
In particular, squares of degree-one classes vanish in $H^2(G,\mathbb F_p)$, also when $p=2$. Thus $T_G$ is alternating.
	
Since $G$ is a pro-$p$ Poincar\'e duality group of dimension three, cup product induces a perfect pairing
\[H^1(G,\mathbb F_p)\times H^2(G,\mathbb F_p) \longrightarrow H^3(G,\mathbb F_p). \]
Composing with $\operatorname{tr}$ gives an isomorphism
\[ \operatorname{PD}_G: H^2(G,\mathbb F_p)\xrightarrow{\sim}V^*,\]
defined by
\[ \operatorname{PD}_G(\alpha)(z) =\operatorname{tr}(\alpha\smile z).\]
Consequently,
\[\operatorname{PD}_G\circ c_G=\mu_G.\]
Since $\operatorname{PD}_G$ is an isomorphism,
\[\ker(c_G)=\ker(\mu_G), \] which proves the asserted identification. The dimension formula follows immediately. As $\dim_{\mathbb F_p}V^*=d$, one has $\operatorname{rank}(\mu_G)\leq d$, giving the stated bounds.
	
Now suppose that $K\ne1$. Then $H^1(K,\mathbb F_p)\ne0$. Choose a nonzero class $\varphi\in H^1(K,\mathbb F_p)$. The conjugation action of the pro-$p$ group $Q$ on continuous homomorphisms $K\to\mathbb F_p$ is continuous when the latter space is given the discrete topology. Indeed, compactness of $K$ ensures that the stabilizer of each continuous homomorphism is open. Thus the $Q$-orbit of $\varphi$ is finite.
	
The vector space spanned by this orbit is a nonzero finite-dimensional $Q$-submodule on which the action factors through a finite $p$-group. Every nonzero representation of a finite $p$-group over $\mathbb F_p$ has a nonzero fixed vector. Therefore
\[H^1(K,\mathbb F_p)^Q\ne0.\]
	
It remains to examine small values of $d$. If $d=2$, then $\bigwedge^3V^*=0$, so $T_G=0$ and $\mu_G=0$. Hence the invariant subspace has dimension one.
	
If $d=3$, the space $\bigwedge^3V^*$ is one-dimensional. Every nonzero alternating trilinear form on a three-dimensional vector space induces an isomorphism
\[\bigwedge\nolimits^2V\xrightarrow{\sim}V^*.\]
Since $\ker(\mu_G)\ne0$, the form $T_G$ must vanish. Consequently $\mu_G=0$, and the invariant subspace has dimension three.
	
Finally, suppose $d=4$. Every nonzero alternating trilinear form on a four-dimensional vector space can be written, after choosing a suitable basis, as
\[T_G=e_1^*\wedge e_2^*\wedge e_3^*.\]
The associated map $\mu_G$ then has rank three. Thus its rank is either zero or three, according as $T_G$ vanishes or not. It follows that
\[\dim_{\mathbb F_p}\ker(\mu_G) =6-\operatorname{rank}(\mu_G) \in\{6,3\}.\]
\end{proof}

\begin{corollary}\label{cor:canonical-kernel-action}
Let $(G,\theta)$ satisfy the hypotheses	of Theorem~4.1 and suppose that $K=K_\theta(G)\ne1$. Put $Q=G/K$ and $d=d(G)$. Then
\[\dim_{\mathbb F_p}H^1(K,\mathbb F_p)=\infty,\]
whereas
\[1\leq \dim_{\mathbb F_p}H^1(K,\mathbb F_p)^Q \leq\binom d2.\]
More precisely, its dimension is $\binom d2-\operatorname{rank}(\mu_G)$, with $\mu_G$ as in Proposition~\ref{prop:canonical-transgression}.
	
The conjugation action of $Q$ on $K/\Phi(K)$ is nontrivial. Furthermore, if $2\leq d\leq4$, then the cohomology algebra $H^\bullet(G,\mathbb F_p)$ is not generated by its degree-one part. In particular, $G$ is not a Bloch--Kato pro-$p$ group.
\end{corollary}

\begin{proof} Theorem~4.1 implies that $K$ is not finitely generated. Hence
\[\dim_{\mathbb F_p}H^1(K,\mathbb F_p)=d(K)=\infty.\]The bounds and the dimension formula follow from Proposition~\ref{prop:canonical-transgression}.
	
If the conjugation action on $K/\Phi(K)$ were trivial, the induced action on $H^1(K,\mathbb F_p)$ would also be trivial. Its fixed subspace would then be infinite-dimensional, contrary to the preceding bound.
	
Finally, degree-two cup products span the image of
\[c_G:\bigwedge\nolimits^2V \longrightarrow H^2(G,\mathbb F_p).\]
Under Poincar\'{e} duality, the rank of this map equals $\operatorname{rank}(\mu_G)$. For $d=2$ and $d=3$, the preceding proposition gives $\mu_G=0$. For $d=4$, its rank is at most three. In every case $2\leq d\leq4$, the cup-product map is therefore not surjective, since
\[\dim_{\mathbb F_p}H^2(G,\mathbb F_p)=d.\] Thus the cohomology algebra is not generated in degree one, and $G$ is not Bloch--Kato.
\end{proof}

\begin{corollary} \label{cor:notcentral}
Under the hypotheses of Corollary~\ref{cor:canonical-kernel-action}, the subgroup $K_\theta(G)$ is not central in $G$.
	\end{corollary}
	
\begin{proof} 
A central kernel would induce the trivial conjugation action on $K_\theta(G)/\Phi(K_\theta(G))$, contrary to Corollary~\ref{cor:canonical-kernel-action}.
\end{proof}
	
\begin{remark}
Proposition~\ref{prop:canonical-transgression} identifies the invariant subspace $\HH^1(K,\Fp)^Q$ with the kernel of the linear map $\mu_G:\bigwedge^2\HH^1(G,\Fp)\to\HH^1(G,\Fp)^*$ induced by the triple cup-product form of $G$. Thus Poincar\'e duality determines the dimension of this invariant subspace from the degree-one cup-product structure, even when $\HH^1(K,\Fp)$ is infinite-dimensional. The remaining module-theoretic problem concerns the action of $Q$ on the full first cohomology of $K$, beyond the information contained in its fixed subspace.
\end{remark}

\section{The finitely generated-kernel case}
\label{sec:classification}
	
When the canonical kernel is finitely generated, the group can be described explicitly.
	
\begin{proposition} \label{prop:classification}
Let $(G,\theta)$ satisfy the hypotheses of Theorem~\ref{thm:main} and assume that $K_\theta(G)$ is finitely generated. Then $G$ is uniform of dimension three and $\theta$-abelian. More precisely: \begin{enumerate}[label=\textup{(\roman*)}]
\item if $\theta$ is trivial, then
\[G\simeq\Zp^3;\]
\item if $\theta$ is nontrivial, then
\[G\simeq\Zp^2\rtimes_\theta\Zp,\]
where the final factor acts on $\Zp^2$ by scalar exponentiation through $\theta$. 
\end{enumerate}
Moreover,
\[\HH^\bullet(G,\Fp)\simeq \bigwedge\nolimits^\bullet\HH^1(G,\Fp)\]
as graded $\Fp$-algebras.
\end{proposition}
	
\begin{proof}
By Theorem~\ref{thm:main}, finite generation of $K_\theta(G)$ implies $K_\theta(G)=1$. Since $(G,\theta)$ is Kummerian, \cite[Theorem~2.2]{QuadrelliHHA} implies that it is $\theta$-abelian. By \cite[Proposition~4.2]{QuadrelliHHA}, the underlying group is locally uniform; finite generation therefore makes $G$ uniform.
		
A uniform pro-$p$ group has cohomological dimension equal to its analytic dimension. Since $G$ is $\PD^3$, 
\[\dim(G)=3.\]If $\theta$ is trivial, then $\theta$-abelianity says simply that $G$ is free abelian pro-$p$, and dimension three gives $G\simeq\Zp^3$.

Suppose now that $\theta$ is nontrivial. The image $\im(\theta)$ is a nontrivial closed torsion-free pro-$p$ subgroup of $1+p\Zp$ (or of $1+4\mathbb Z_2$ when $p=2$), hence isomorphic to $\Zp$. We have an exact sequence 
\[1\longrightarrow\ker(\theta) \longrightarrow G \longrightarrow\im(\theta)
\longrightarrow 1.\]
The kernel is free abelian pro-$p$, and $\im(\theta)\simeq\Zp$ is a
free pro-$p$ group of rank one, hence projective. The extension therefore splits. Since $\dim(G)=3$, the kernel has rank two. The $\theta$-abelian relations show that the action of the final $\Zp$ factor on $\Zp^2$ is scalar exponentiation through $\theta$. 

Finally, the cohomology statement follows from Lazard's theorem for uniform pro-$p$ groups. 
\end{proof}
	
\begin{corollary} \label{cor:BK}
Under the hypotheses of Proposition~\ref{prop:classification}, $G$ is Bloch--Kato, and hence weakly Bloch--Kato. \end{corollary}
	
\begin{proof}
Finitely generated locally uniform pro-$p$ groups are Bloch--Kato; see the discussion following \cite[Proposition~4.2]{QuadrelliHHA}. The weak Bloch--Kato property follows immediately. 
\end{proof}

\section{Examples, scope, and the remaining regime}\label{sec:examples}

Theorem~\ref{thm:main} distinguishes two possibilities for the canonical Kummerian kernel of a finitely generated 1-smooth pro-$p$ $\PD^3$ group: the kernel is trivial, in which case the group is $\theta$-abelian and uniform, or it is infinitely generated. We first describe the two forms of the $\theta$-abelian case, and then exhibit a non-$\theta$-abelian example with a free canonical kernel of infinite rank.

\begin{example}[Trivial orientation]
Let \[G=\Zp^3\]
with the trivial orientation. Then $(G,\theta)$ is 1-smooth, and $G$ is uniform of dimension three and hence is a pro-$p$ $\PD^3$ group. Since $G$ is abelian,
\[K_\theta(G)=1.\] Its mod-$p$ cohomology algebra is
\[\HH^\bullet(G,\Fp)\simeq\bigwedge\nolimits^\bullet\Fp^3.\]
\end{example}

\begin{example}[Nontrivial scalar orientation]
Let $A\simeq\Zp^2$ and $T\simeq\Zp$. Choose a continuous injective homomorphism
\[\vartheta:T\longrightarrow 1+p\Zp,\]
	with image contained in $1+4\mathbb Z_2$ when $p=2$. Let $T$ act on $A$ by scalar exponentiation through $\vartheta$, and set
	\[G=A\rtimes_\vartheta T.\]
	Equip $G$ with the orientation induced by $\vartheta$. The pair $(G,\theta)$ is $\theta$-abelian and 1-smooth, and therefore
	\[K_\theta(G)=1.\]
	Moreover, $G$ is uniform of dimension three and hence is a pro-$p$ $\PD^3$ group. Its mod-$p$ cohomology algebra is again
	\[\HH^\bullet(G,\Fp)\simeq\bigwedge\nolimits^\bullet\Fp^3.\]
\end{example}

The next example realizes the complementary case of Theorem~\ref{thm:main}. Its construction uses the cyclotomic semidirect products studied by Quadrelli and Weigel \cite{QuadrelliWeigelBP}, and the freeness of its canonical kernel follows from their results for Demu\v{s}kin groups.

\begin{example}[A nontrivial canonical kernel of infinite rank]\label{ex:Demu\v{s}kin-extension}
Let $p=3$, and let $F=\mathbb Q_3(\zeta_3)$, where $\zeta_3$ is a primitive third root of unity. Write $D=G_F(3)$ for the maximal pro-$3$ Galois group of $F$, and let $\chi:D\longrightarrow 1+3\mathbb Z_3$ be its cyclotomic orientation.
	
By the classical description of maximal pro-$p$ Galois groups of local fields, $D$ is a Demu\v{s}kin pro-$3$ group with
\[d(D)=[F:\mathbb Q_3]+2=4,\qquad \cdp(D)=2.\]
The cyclotomic orientation $\chi$ coincides with the canonical Demu\v{s}kin orientation, and $(D,\chi)$ is 1-smooth.
	
Consider the cyclotomic semidirect product $G=\mathbb Z_3\rtimes_\chi D$, where $D$ acts on $\mathbb Z_3$ by multiplication through $\chi$. Let $\pi:G\longrightarrow D$ be the natural projection and equip $G$ with the orientation $\theta=\chi\circ\pi$.
	
The valuation-theoretic description of maximal pro-$3$ Galois groups of Laurent series fields gives an isomorphism of oriented pro-$3$ groups
\[(G,\theta)\simeq \bigl(G_{F((T))}(3),\chi_{F((T))}\bigr).\]
Indeed, since the residue field of $F((T))$ has characteristic zero, its pro-$3$ inertia group is isomorphic to $\mathbb Z_3(1)$. The coefficient-field embedding provides a splitting of the corresponding Galois sequence, and conjugation on inertia is given by the cyclotomic character. In particular, $(G,\theta)$ is 1-smooth by the Kummer-theoretic lifting property of maximal pro-$3$ Galois groups. The construction also belongs to the elementary-type family described by Quadrelli--Weigel \cite[Section~5.3 and Proposition~5.13(f)]{QuadrelliWeigelBP}.
	
The group $G$ fits into a split exact sequence
\[1\longrightarrow\mathbb Z_3\longrightarrow G\longrightarrow D\longrightarrow 1.\]
Since $\mathbb Z_3$ is a pro-$3$ $\PD^1$ group and $D$ is a pro-$3$ $\PD^2$ group, the extension theorem for profinite Poincar\'e duality groups \cite[Theorem~3.7.4]{NSW} implies that $G$ is a pro-$3$ $\PD^3$ group.
	
The canonical kernel of this extension can be identified explicitly. By Quadrelli--Weigel \cite[Proposition~5.9]{QuadrelliWeigelBP},
\[K_\theta(G)=K_\chi(D),\]where the subgroup on the right is identified with its image in the canonical copy of $D$ in $G$. This kernel is nontrivial. Indeed, if $K_\chi(D)=1$, then Kummerianity would imply that $(D,\chi)$ is $\chi$-abelian. Consequently $D$ would be uniform, and hence
\[\cdp(D)=\dim(D)=d(D)=4,\]	contrary to $\cdp(D)=2$.
	
Quadrelli--Weigel \cite[Theorem~5.1]{QuadrelliWeigelBP} prove that the canonical kernel of a Demu\v{s}kin group with its canonical torsion-free orientation is free pro-$p$. Therefore $K_\theta(G)$ is free pro-$3$. Since it is nontrivial, Theorem~\ref{thm:main} shows that it cannot be finitely generated. Thus $K_\theta(G)$ is a free pro-3 group of infinite rank. Moreover, $G$ is not $3$-adic analytic: otherwise the analytic 1-smoothness theorem \cite[Theorem~5.7]{QuadrelliHHA} would force $G$ to be $\theta$-abelian, contradicting $K_\theta(G)\neq1$.
	
The mod-$3$ cohomology of $G$ can also be calculated explicitly. Since $\chi$ is trivial modulo $3$, the map
\[G=\mathbb Z_3\rtimes_\chi D\longrightarrow\mathbb F_3,\qquad (a,g)\longmapsto a \bmod 3\]
is a continuous homomorphism. It determines a class $x\in\HH^1(G,\mathbb F_3)$ whose restriction generates $\HH^1(\mathbb Z_3,\mathbb F_3)$. The Hochschild--Serre spectral sequence for the split extension has only two potentially nonzero rows, corresponding to $\HH^0(\mathbb Z_3,\mathbb F_3)$ and $\HH^1(\mathbb Z_3,\mathbb F_3)$. The class $x$ lifts the generator of the second row, while inflation identifies the first row with $\HH^\bullet(D,\mathbb F_3)$. Multiplication by $x$ identifies the two rows, so the spectral sequence collapses multiplicatively. Since $3$ is odd, graded commutativity gives $x^2=0$, and consequently
\[\HH^\bullet(G,\mathbb F_3)\simeq\HH^\bullet(D,\mathbb F_3)\otimes_{\mathbb F_3}\bigwedge(x)\] as graded $\mathbb F_3$-algebras.
	
In particular, \[d(G)=d(D)+1=5.\]
Put $W=\HH^1(D,\mathbb F_3)$. Then
\[\HH^1(G,\mathbb F_3)=W\oplus\mathbb F_3x.\]Since $D$ is Demu\v{s}kin, cup product induces a nondegenerate alternating pairing
\[W\times W\longrightarrow\HH^2(D,\mathbb F_3)\simeq\mathbb F_3.\]
Under the preceding cohomological decomposition, we have
\[\HH^2(G,\mathbb F_3)\simeq\HH^2(D,\mathbb F_3)\oplus\bigl(W\smile x\bigr).\]
Products of two elements of $W$ span the first summand, while products of $x$ with elements of $W$ span the second. Hence the cup-product map
\[\bigwedge\nolimits^2\HH^1(G,\mathbb F_3)\longrightarrow\HH^2(G,\mathbb F_3)\] is surjective. By Poincar\'e duality, the map $\mu_G$ of Proposition~\ref{prop:canonical-transgression} is therefore surjective as well, and
\[\operatorname{rank}(\mu_G)=\dim_{\mathbb F_3}\HH^2(G,\mathbb F_3)=5.\]
Proposition~\ref{prop:canonical-transgression} now yields
\[\dim_{\mathbb F_3}\HH^1(K_\theta(G),\mathbb F_3)^{G/K_\theta(G)}=\binom52-5=5.\] Thus $G$ is a nonanalytic 1-smooth pro-$3$ $\PD^3$ group with a nontrivial free canonical kernel of infinite rank. Its first kernel cohomology is infinite-dimensional, whereas its invariant subspace has dimension five, determined explicitly by the triple cup-product form. Moreover, its mod-$3$ cohomology is generated in degree one with quadratic relations, and $G$ is Bloch--Kato as the maximal pro-$3$ Galois group of a field containing $\mu_3$. The example therefore realizes the infinite-kernel alternative of Theorem~\ref{thm:main} while satisfying the cohomological restrictions of Proposition~\ref{prop:canonical-transgression}.
\end{example}

The three examples illustrate both alternatives in Theorem~\ref{thm:main}. The first two describe the $\theta$-abelian possibilities, while Example~\ref{ex:Demu\v{s}kin-extension} shows that non-$\theta$-abelian 1-smooth pro-$p$ $\PD^3$ groups occur and may have free canonical kernels of infinite rank. The finite-generation obstruction therefore separates two genuinely realized forms of canonical kernel behavior.

More generally, let $(G,\theta)$ satisfy the hypotheses of Theorem~\ref{thm:main} and suppose that $G$ is not $\theta$-abelian. Then
\begin{equation}
1\neq K_\theta(G)\subseteq\Phi(G),\;\; d(K_\theta(G))=\infty,\qquad \cdp(K_\theta(G))\leq2. \label{eq:remaining}
\end{equation}
Moreover, the quotient
\[Q=G/K_\theta(G)\] is uniform, has $d(Q)=d(G)$, and preserves degree-one mod-$p$ cohomology under inflation.

The Hochschild--Serre sequence and Poincar\'{e} duality give a more precise description of the canonical extension. Writing $K=K_\theta(G)$ and $V=\HH^1(G,\Fp)$, Proposition~\ref{prop:canonical-transgression} identifies
\[\HH^1(K,\Fp)^Q\simeq\ker\left(\bigwedge\nolimits^2V\xrightarrow{\mu_G}V^*\right),\]
where $\mu_G$ is induced by the triple cup-product form of $G$. Consequently,
\[\dim_{\Fp}\HH^1(K,\Fp)^Q=\binom{d(G)}2-\operatorname{rank}(\mu_G).\]
Since $K\neq1$, this invariant subspace is nonzero. Thus
\[1\leq\dim_{\Fp} \HH^1(K,\Fp)^Q\leq\binom{d(G)}2,\]
whereas
\[\dim_{\Fp }\HH^1(K,\Fp)=\infty.\]
In particular, the quotient $Q$ acts nontrivially on $K/\Phi(K)$, and $K$ cannot be central in $G$.

The Bogomolov--Positselski property provides a natural strengthening of these conclusions. Following Quadrelli and Weigel \cite{QuadrelliWeigelBP}, a Kummerian oriented pro-$p$ group has this property when its canonical kernel is free pro-$p$. Their results establish the property for Demu\v{s}kin groups with their canonical orientations and preserve it under cyclotomic semidirect products with free abelian pro-$p$ groups. Example~\ref{ex:Demu\v{s}kin-extension} belongs to this class. For an arbitrary 1-smooth pro-$p$ $\PD^3$ group, Theorem~\ref{thm:main} establishes the infinite-generation alternative, while freeness remains a further structural question.

The remaining questions concern the freeness of canonical kernels, the module structure of their first cohomology, and the extent to which 1-smoothness controls quadraticity in dimension three.

\begin{question}\label{ques:kernel-freeness}
Let $(G,\theta)$ be a finitely generated 1-smooth pro-$p$ $\PD^3$ group with torsion-free orientation. If $K_\theta(G)\neq1$, must $K_\theta(G)$ be a free pro-$p$ group?
\end{question}

A positive answer would establish the Bogomolov--Positselski property throughout this class. Theorem~\ref{thm:main} shows that every nontrivial canonical kernel would then be free of infinite rank, as already occurs in Example~\ref{ex:Demu\v{s}kin-extension}.

\begin{question}\label{ques:kernel-module}
	Let $(G,\theta)$ be a finitely generated 1-smooth pro-$p$ $\PD^3$ group with torsion-free orientation, and assume that $K=K_\theta(G)\neq1$. Put $Q=G/K$ and $V=\HH^1(G,\Fp)$. To what extent is the structure of the discrete $\Fp[[Q]]$-module
	\[\HH^1(K,\Fp)\]
	determined by the triple cup-product form
	\[T_G:\bigwedge\nolimits^3V\longrightarrow\Fp?\]
	In particular, does the triple cup-product form, together with the identification $V\simeq\HH^1(Q,\Fp)$, determine the isomorphism class of this module, or can distinct canonical extensions with the same quotient and triple cup-product form yield nonisomorphic modules?
\end{question}

Proposition~\ref{prop:canonical-transgression} determines the fixed subspace of $\HH^1(K,\Fp)$ from $T_G$. The question asks whether Poincar\'e duality and 1-smoothness impose further constraints on the full infinite-dimensional module, beyond those detected by degree-two inflation and the five-term exact sequence.

\begin{question}\label{ques:weak-bk}
Is every finitely generated 1-smooth pro-$p$ $\PD^3$ group weakly Bloch--Kato?
\end{question}

Theorem~\ref{thm:main}, together with the classification of the $\theta$-abelian case, gives an affirmative answer whenever the canonical kernel is finitely generated. Example~\ref{ex:Demu\v{s}kin-extension} gives a nonanalytic instance with an infinitely generated canonical kernel for which the Bloch--Kato property also holds. The general question therefore concerns the relationship between the structure of infinite canonical kernels and the cohomological restrictions imposed by 1-smoothness.

\end{document}